\documentclass[draft,10pt,a4paper]{article}
\pdfoutput=1

\usepackage{a4wide}
\usepackage[english]{babel}
\usepackage[T1]{fontenc}
\usepackage[utf8]{inputenc}
\usepackage{amssymb}
\usepackage{amsthm}
\usepackage{amsmath}
\usepackage{url,verbatim}
\usepackage{longtable}

\newtheorem{theorem}{Theorem}

\title{Constructing large tables of numbers of maps\\
 by orientable genus}

\author{
Alain Giorgetti$^1$ $\qquad$ Timothy~R.~S. Walsh$^2$ \\
\vspace{0.4mm}
$^1$ FEMTO-ST Institute, University of Franche-Comt\'e,\\
 16 route de Gray, 25030 Besan\c{c}on CEDEX, France,\\
\texttt{alain.giorgetti@femto-st.fr}\\
\vspace{0.4mm}
$^2$ Department of Computer Science,\\
University of Quebec in Montreal (UQAM),\\
P.O.~Box~8888, Station A, Montreal, Quebec, Canada, HC3-3P8,\\
\texttt{walsh.timothy@uqam.ca} 
}
 
\date{}

\begin{document}
\maketitle

\begin{abstract}
The Carrell-Chapuy recurrence formulas dramatically improve the efficiency of
counting orientable rooted maps by genus, either by number of edges alone or by
number of edges and vertices. This paper presents an implementation of these
formulas with three applications: the computation of an explicit rational
expression for the ordinary generating functions of rooted map numbers with a
given positive genus, the construction of large tables of rooted map numbers,
and the use of these tables, together with the method of A. Mednykh and R.
Nedela, to count unrooted maps by genus and number of edges and vertices.
\end{abstract}

\section{Introduction}
A \emph{map} is a $2$-cell imbedding of a connected graph, loops and multiple
edges allowed, on a compact surface, which in this article will be taken to be
orientable and without boundaries, and is thus characterized by a single
non-negative integer, its \emph{genus}. A map is \emph{rooted} if a \emph{dart}
-- an edge-vertex incidence pair -- is distinguished as the root. By
\emph{counting} maps we mean counting equivalence classes of maps under
orientation-preserving homeomorphism; in the case of rooted maps, the
homeomorphism must preserve the distinguished oriented edge. In this case the
homeomorphism preserves all the darts~\cite{Tu1}, so that rooted maps can be
counted without considering the symmetries of the maps, which is why rooted maps
were counted before unrooted maps.

Let $m_g(n)$ be the number of rooted maps with $n$ edges on the orientable
surface of genus $g$. Let $M_g(z) = \sum_{n \geq 0} {m}_g(n) z^n$ be the
ordinary generating function counting genus-$g$ rooted maps by number of edges
(the exponent of $z$). Let $m_g(v,f)$ be the number of rooted genus-$g$ maps
with $v$ vertices and $f$ faces. By face-vertex duality, this number is equal to
the number $m_g(f,v)$ of rooted genus-$g$ maps with $f$ vertices and $v$ faces.
The ordinary generating function that counts rooted genus-$g$ maps is the
following formal power series in two variables $u$ and $w$:
\begin{equation}
M_g(w,u) = \sum_{v,f \geq 1} m_g(v,f) w^v u^f. \label{Mg:eq}
\end{equation}

The Carrell-Chapuy recurrence formulas~\cite{CC14-v3} dramatically improve the
efficiency of counting rooted maps by genus. We show how to use them
 to determine explicit rational expressions for the generating functions
 $M_g(z)$ and closed-form formulas for the numbers $m_g(n)$.
We also have used Carrell-Chapuy recurrence formulas to construct large tables
of numbers of rooted and unrooted maps of genus up to 50 with up to 100 edges.
Our goal is to provide these numbers to researchers for further studies of their
 properties.

The paper is organized as follows. Section~\ref{history:sec} summarizes the
history of two closely related problems, namely computing numbers of rooted maps
by genus and finding a closed form for their generating functions.
Section~\ref{rm:edge:sec} presents formulas for numbers of rooted maps with a
fixed genus. In Section~\ref{unrooted:sec} we discuss counting unrooted maps and
in Section~\ref{complex:sec} we give a complexity analysis and the results of
time trials. In the appendix we include a table of numbers of unrooted
maps counted by genus, number of edges and number of vertices.  We do not
include a table of numbers of rooted maps because the reader can easily
construct such a table from the recurrence in~\cite[Corollary 3]{CC14-v3}  or
the optimized version of it that we present as formula (\ref{cc:edge:face:th})
here.  A larger table, a table of numbers of rooted maps and
a text file of the source code are available from the second author on request
and can be found in release 0.4.0 of the MAP project~\cite{MAP}. The source code can
also be found in~\cite{Walsh14}.

\section{Historical notes}
\label{history:sec}
For counting by number of edges alone, W.~T. Tutte~\cite{Tu1} first showed that
the generating function $M_0(z)$ for rooted planar maps can be parametrically
defined by $M_0(z) = (3-\xi)(\xi-1)/3$, where the parameter $\xi$ is the series
in $z$ satisfying $\xi = 1 + 3z \xi^2$. Tutte also found a closed-form formula
for $m_0(n)$. For counting with two parameters (i.e. by number of edges and
vertices, edges and faces, or vertices and faces), W.~T. Tutte~\cite{Tu3} and D.
Arqu\`es~\cite[Theorem 4]{Arq87a} respectively found a parametric polynomial
definition of $M_0(w,u)$ and
 a parametric rational definition of
$M_1(w,u)$.
Arqu\` es also obtained a closed-form formula for the number of rooted toroidal
maps with $n$ edges and another one for the number of rooted toroidal maps with
$v$ vertices and $f$ faces.

In \cite{Tu3}, a recursive formula was found for the number of rooted planar
maps given the number of vertices, the number of edges, and the degree of the
face containing the root; these numbers of maps were then added over all
possible degrees of this face and the result expressed in terms of generating
functions. In~\cite{Walsh71}, this method was generalized to obtain a recursive
formula for the number of maps of genus $g$ with a distinguished dart in each
vertex given the number of vertices and the degree of each one; these numbers
were then multiplied by the appropriate factor and added over all possible
non-increasing sequences of vertex-degrees summing to $2n$ to obtain the number
of rooted maps of genus $g$ with $n$ edges and $v$ vertices.  A table of these
numbers of maps with up to $14$ edges appears in~\cite{Walsh71} (see~\cite{W1}
for a published account of this work and a table of maps with up to $11$ edges)
but no attempt was made there to express this result in terms of generating
functions. 
In~\cite{BC86} an improvement on the method of \cite{Walsh71} was introduced:
to count rooted genus-$g$ maps it is sufficient to know the degree of the first
$g+1$ vertices and to distinguish a dart of only the first vertex as the root,
thus reducing the number of maps that have to be considered.  

For any genus $g$ the existence of a parametric rational expression for the
generating functions $M_g(z)$ and $M_g(w,u)$ is stated by E. Bender and E.
Canfield, in~\cite{BC91} for $M_g(z)$ and in~\cite{BCR93} for  $M_g(w,u)$. The
first of these two papers~\cite{BC91} also presents explicit rational functions
for $M_2(z)$ and $M_3(z)$. A common pattern for all these rational functions
is proposed, and an upper bound for the degree of their numerator is
conjectured. For the univariate function $M_g(z)$ a bound was found in
\cite{gio98a} and refined in~\cite{giorgetti10:ip}. For the bivariate function
$M_g(w,u)$ a bound was proved in~\cite{ag99}. These results are summarized in
the following two theorems.

\begin{theorem}[\cite{giorgetti10:ip}]
\label{pattern:th}
For any positive integer $g$, the ordinary generating function $M_g(z)$ counting
rooted maps on a closed orientable surface of genus $g$ by number of edges
(exponent of $z$) can be written as

\begin{equation*}
M_g(z) = z^{2g} (1-2m)^{2-3g} (1-3 m)^{-2} (1-6m)^{3-5g} P_g(m),
\end{equation*}
where $m = \dfrac{1 - \sqrt{1-12z}}{6}$ and $P_g(m)$ is a
polynomial of $m$ of degree at most $4g-4$.
\end{theorem}

\begin{theorem}[{\cite[Theorem 1]{ag99}}]
\label{pattern:vertices:faces:th}
For any positive integer $g$, the ordinary generating function $M_g(w,u)$
counting rooted maps on a closed orientable surface of genus $g$ by number of
vertices (exponent of $w$) and faces (exponent of $u$) can be
written as
\begin{eqnarray}
M_{{g}}(w,u)=
{\frac 
  {pq\left (1-p-q\right ) P_{g}(p,q)}
  {{\left[\left (1-2p-2q\right )^{2}-4pq\right]}^{5g-3}
  },
}                  \label{MgP:eq}
\end{eqnarray}
where $P_{g}(p,q)$ is a symmetric polynomial in $p$ and $q$ 
of total degree at most $6g-6$ with integral coefficients.
\end{theorem}

In~\cite{WG14} the first and second authors calculated the polynomial $P_g(p,q)$
for $g$ up to $6$ and thus counted rooted maps of genus up to $6$ by number of
vertices and faces as well as by number of edges (using
Theorem~\ref{pattern:th}). In~\cite{WGM12} the first author of that paper, using
a more powerful computer, extended these calculations up to genus 10 and also
counted unrooted maps of genus up to 10 by number of vertices and faces, the
second author counted rooted maps of genus up to 11 by number of edges and the
third author, A. Mednykh, counted unrooted maps of genus 11 by number of edges.
The cost of counting unrooted maps, once a table of numbers of rooted maps has
been constructed, was greatly dominated by the cost of counting rooted maps.

More recently, a far more efficient method for counting rooted maps with a fixed
genus was discovered by S. R. Carrell and G. Chapuy~\cite{CC14-v3}. They
showed~\cite[Theorem 1]{CC14-v3} that the number $m_g(n)$ of rooted maps of
genus $g$ with $n$ edges satisfies the following recurrence relation (we have
modified the formulas for the sake of computational efficiency):

\begin{eqnarray}
(n + 1)\, m_g(n) & = & (8n-4)\, m_g(n-1) \nonumber \\
 & & + \ (2n-3) (n-1) (2n-1) \, m_{g-1}(n-2) \label{cc:edge:th} \\
& &
+ \ 3
\sum_{
  \begin{subarray}{c}
   i+j=g\\ 
   i,j\geq 0
  \end{subarray}
 }
 \sum_{
  \begin{subarray}{c}
   k+l=n-2\\
   k \geq 2i, l \geq 2j \\
  \end{subarray}
 } (2k+1) (2l+1) \, m_i(k) \, m_j(l) \nonumber
\end{eqnarray}
for $n \geq  1$, with the initial conditions $m_0(0) = 1$ and $m_g(n) = 0$ if $g
< 0$ or $n < 2g$. For counting with two parameters, Carrell and Chapuy
showed~\cite[Corollary 3]{CC14-v3} that the number $m_g(n,f)$ of rooted maps of
genus $g$ with $n$ edges and $f$ faces satisfies the following recurrence
relation:
\begin{eqnarray}
(n + 1) \, m_g(n,f) & = & (4n - 2) (m_g(n-1,f)+m_g(n-1,f-1)) \nonumber \\
 & & + \ (2n - 3)(n - 1)(2n - 1) \, m_{g-1}(n-2,f) \label{cc:edge:face:th} \\
& & + \ 3
\sum_{
    \begin{subarray}{c}
    i+j=g\\ 
    i,j\geq 0
    \end{subarray}
 }
 \sum_{
    \begin{subarray}{c}
    k+l=n-2\\
     k\geq 2i, l \geq 2j 
    \end{subarray}
 }
 \sum_{
    \begin{subarray}{c}
    u+v=f\\
    u,v \geq 1
    \end{subarray}
 }
  (2k+1)(2l+1) \, m_{i}(k,u) \, m_{j}(l,v) \nonumber
\end{eqnarray}
for $n, f \geq  1$, with the initial conditions $m_{0}(0,1) = 1$ and $m_{g}(n,f)
= 0$ if $g < 0$ or $n < 2g$ or $f < 1$ or $n-f+2(1-g) < 1$.

\section{Fixed genus formulas}
\label{rm:edge:sec}

This section shows simple -- but as far as we know not yet published --
consequences of Theorem~\ref{pattern:th} (about the rationality of the
generating series $M_g(z)$) for the computation of rooted map numbers.
The first consequence, given in Theorem~\ref{fixed:genus:edge:rec}, is a
recurrence formula between numbers of rooted maps with the same positive genus
$g$. The second consequence, given in Theorem~\ref{closed:genus:edge:nb}, is a
closed formula for the number $m_g(n)$, for any positive genus $g$ and any
number of edges $n$. This formula depends on integers which are the coefficients
of the polynomial $P_g(m)$ from Theorem~\ref{pattern:th}. We start this section
with a theorem completing Theorem~\ref{pattern:th} with an explicit relation
between this polynomial and the numbers of rooted maps with the same genus and
up to $6g-4$ edges. In this section $[x^n] S(x)$ denotes the coefficient of
$x^n$ in the formal power series $S(x)$. By convention, a sum over an empty
domain is equal to zero.

\begin{theorem}
\label{explicit:rational:edge:th}

For any positive integer $g$, the ordinary generating function $M_g(z)$ counting
rooted maps on a closed orientable surface of genus $g$ by number of edges
(exponent of $z$) is

\begin{equation}
M_g(z) = z^{2g} P_g(m) / F_g(m),
\label{qg:pattern:eq}
\end{equation}
where $m = \dfrac{1 - \sqrt{1-12z}}{6}$ and
$F_g(m)$ and $P_g(m)$ are the polynomials defined by 
\begin{equation}
F_g(m) = (1-2m)^{3g-2} (1-3m)^{2} (1-6m)^{5g-3} \label{Fg:def:eq}
\end{equation}
and
$P_g(m) = \sum_{l \leq 4g-4} p_{g,l} m^l$
with
\begin{equation}
p_{g,l} = \sum_{n =2g}^{6g-4} (-1)^{l-n} m_g(n) 
\sum_{i+j+k=l-n+2g}
 2^{i+k} 3^{j+k} {3g-2 \choose i}  {n-2g+2 \choose j}  {5g-3 \choose k}. 
 \label{pol:coeff:edge}
\end{equation}
\end{theorem}

\begin{proof}
On the one hand, $M_g(z) = \sum_{n\geq 2g} m_g(n) z^n$
since $m_g(n) = 0$ if $n < 2g$. On the other hand, Theorem~\ref{pattern:th}
gives
$P_g(m) = z^{-2g} F_g(m) M_g(z)$,
where $z=m(1-3m)$. From both of these resuts we obtain 
\begin{equation}
P_g(m) = F_g(m) \sum_{n \geq 2g} m_g(n) (m(1-3m))^{n-2g}.
\label{Mg:edge:eq4}
\end{equation}
Since the degree of the polynomial $P_g(m)$ is at most $4g-4$, 
$P_g(m) = \sum_{l \leq 4g-4} p_{g,l} m^l$
with
$$
p_{g,l} = [m^l] P_g(m)
  =  \sum_{n \geq 2g}^{6g-4} m_g(n) [m^{l-n+2g}] (1-2m)^{3g-2} (1-3m)^{n-2g+2} 
 (1-6m)^{5g-3}.
$$
From $(1-am)^k = \sum_{i=0}^{k} {k \choose i} (-a)^{i} m^i$, it follows that
\begin{equation*}
p_{g,l} = \sum_{n =2g}^{6g-4} m_g(n)
\sum_{i+j+k=l-n+2g}
 {3g-2 \choose i} (-2)^i  
 {n-2g+2 \choose j} (-3)^j 
 {5g-3 \choose k} (-6)^k, 
\end{equation*}
which implies Formula (\ref{pol:coeff:edge}).
\end{proof}

The polynomial $P_1(m)$ can be derived from~\cite{Arq87a} and the polynomials
$P_2(m)$ and $P_3(m)$ from~\cite{BC91}. The polynomial $P_4(m)$ was first given
in~\cite{MG11}. All the polynomials $P_g(m)$ with $1 \leq g \leq 6$ can be found
in~\cite[Appendix B]{WG14}. They were computed by a complicated recurrence
formula involving additional parameters.
Theorem~\ref{explicit:rational:edge:th} and Carrell-Chapuy recurrence formula
(\ref{cc:edge:th}) provide a much more efficient way to compute the polynomials
$P_g(m)$.

\begin{theorem}
\label{fixed:genus:edge:rec}
For any positive integer $g$, the number $m_g(n)$ of rooted maps of
positive genus $g$ with $n$ edges  is recursively defined for $n \geq 6g-3$ by
\begin{equation}
m_g(n)  = \sum_{e = 2g}^{n-1} (-1)^{n-e-1} m_g(e)
            \sum_{i+j+k=n-e}  2^{j+k} 3^{i+k} {e-2g+2 \choose i}
            {3g-2 \choose j} {5g-3 \choose k}.
\label{mg:edge:rec}
\end{equation}
\end{theorem}

\begin{proof}
 For $n \geq 4g-3$ it follows from (\ref{Mg:edge:eq4}) and the fact
that the degree of the polynomial $P_g(m)$ is at most $4g-4$  that
\begin{equation}
[m^n] \left( F_g(m) \sum_{e \geq 2g} m_g(e) (m(1-3m))^{e-2g}\right) = 0
\label{edge:eq5}
\end{equation}
i.e. 
\begin{equation}
 \sum_{e = 2g}^{n+2g} m_g(e) \, [m^{n-e+2g}] \left( 
  (1-2m)^{3g-2}  (1-3m)^{e-2g+2} (1-6m)^{5g-3} 
 \right) = 0.
\label{edge:eq7}
\end{equation}
By isolating $m_g(n+2g)$ in the left-hand side, we obtain
\begin{equation}
  m_g(n+2g) =  - \sum_{e = 2g}^{n+2g-1} m_g(e) \, [m^{n-e+2g}] \left( 
  (1-2m)^{3g-2}  (1-3m)^{e-2g+2} (1-6m)^{5g-3} 
 \right)
\label{edge:eq8}
\end{equation}
 and Formula (\ref{mg:edge:rec}) follows from $(1-am)^k = \sum_{i=0}^{k} {k
 \choose i} (-a)^{i} m^i$.
\end{proof}

Theorems~\ref{explicit:rational:edge:th} and~\ref{fixed:genus:edge:rec} also
imply that the series $(m_g(n))_{n \geq 0}$ of rooted maps of positive genus $g$
is uniquely determined by its first $6g-4$ values (among which the first $2g$
values are known to be $0$).

\begin{theorem}
\label{closed:genus:edge:nb}
For any positive integer $g$, the number $m_g(n)$ of rooted maps of
positive genus $g$ with $n$ edges is defined by the closed formula
\begin{equation}
 m_g(n) = \sum_{l=0}^{4g-4} p_{g,l} 
  \sum_{i+j+k=n-2g-l}  2^{i+k} 3^{j+k} 
   {i+3g-3 \choose i}   {j+n-2g+2 \choose j}  {k+5g-5 \choose k}  
   \label{closed:edge:formula}
\end{equation}
where the integers $p_{g,l}$ are defined in
Theorem~\ref{explicit:rational:edge:th}.
\end{theorem}

\begin{proof}
By applying the Lagrange inversion formula~\cite[page 38]{Stanley2} to
(\ref{qg:pattern:eq}), we obtain 

$$[z^{n-2g}] \frac{M_g(z)}{z^{2g}} = [m^{n-2g}]
\frac{P_g(m) (1-6m)}{F_g(m) (1-3m)^{n-2g+1}}$$ i.e.
$$m_g(n) = [m^{n-2g}] \frac{P_g(m)}{(1-2m)^{3g-2} (1-3m)^{n-2g+3}
(1-6m)^{5g-4}}.$$
From $(1-x)^{-p} = \sum_{k \geq 0} {k+p-1 \choose k} x^k$ for $p \geq 0$, it
follows that
$$m_g(n) = \sum_{l=0}^{4g-4} p_{g,l} 
  \sum_{i+j+k=n-2g-l} 
  {i+3g-3 \choose i} 2^i 
  {j+n-2g+2 \choose j} 3^j 
  {k+5g-5 \choose k} 6^k$$
which implies Formula (\ref{closed:edge:formula}).
\end{proof}

As far as we know, these formulas are the first ones relating numbers of
genus-$g$ rooted maps only with other numbers of rooted maps of the same genus.
These formulas can be easily generalized to bivariate generating functions.
We did not use them for counting unrooted maps because they are
computationally less efficient than (\ref{cc:edge:th}).

\section{Counting unrooted maps}
\label{unrooted:sec}

The second author has written a program in C++ that counts rooted maps by genus,
number of edges and number of vertices using Formula (\ref{cc:edge:face:th})
(an optimized version of~\cite[Corollary 3]{CC14-v3}) and also counts
unrooted maps with these parameters.
A table of numbers and a text file of the source code are available  from the
second author on request and in release 0.4.0 of the MAP project~\cite{MAP}. The
source code can also be found in~\cite{Walsh14}. The method used to count
unrooted maps given a table of numbers of rooted maps was presented by A.
Mednykh and R. Nedela in~\cite{MN06} and refined by V. Liskovets
in~\cite{Liskovets10}. The second author used Liskovets' method. Details of the
calculations were presented in~\cite{WGM12}; a more pedagogical exposition can
be found in~\cite{W13}. For the sake of brevity we do not include the details
here. Suffice it to say that the program uses the software CLN to handle big
integers and the C/C++ compiler XCODE to run CLN; both of these software
packages can be downloaded free of charge from the internet.

\section{Complexity analysis and time trials}
\label{complex:sec}
Let $n$ be the largest number of edges in the maps to be enumerated. Then the
maximum genus of the maps is $\lfloor\frac{n}{2}\rfloor$, since a map of genus
$g$ must have at least $2g$ edges. The recurrence in~\cite{CC14-v3} uses
$O(n^2)$ arithmetic operations to obtain the number of rooted maps with at most
$n$ edges, counted by number of edges alone, and $O(n^3)$ arithmetic operations
if one is counting by number of vertices as well as edges. To construct a table
of numbers of rooted maps with up to $n$ edges takes $O(n^4)$ arithmetic
operations if we are counting by number of edges alone, since there are $O(n^2)$
entries in the table, and $O(n^6)$ arithmetic operations if we are counting by
number of vertices as well, since there are $O(n^3)$ entries in the table.

These asymptotic estimates are overoptimistic if we are considering CPU time
rather than number of arithmetic operations because we are working with big
numbers. The arithmetic operation executed most often in counting rooted maps is
multiplying one number of rooted maps by another one. The naïve method of
multiplying two numbers of size (number of bits) bounded by $b$ takes $O(b^2)$
bit operations. When the numbers get big enough, CLN uses the Schönhage-Strassen
method of multiplication~\cite{SS71}, which uses $O(b~\log~b~\log~\log~b)$ bit
operations. In~\cite{BC86} it was shown that the number of rooted maps of genus
$g$ with $n$ edges is asymptotic to $12n$ multiplied by a polynomial in $n$
whose degree depends linearly on $g$. It follows that the number of bits in such a
number is in $O(n + g~\log n)$, and since $g$ is in $O(n)$, the size of the
number is in $O(n \log n)$. The true cost of counting rooted maps with up to $n$
edges is found by multiplying the number of arithmetic operations by $O(n^2 (\log
 n)^2)$ if naïve multiplication is used and by $O(n(\log n)^2 \log \log n)$ if the
 Schönhage-Strassen method is used.

The second author's program calculates $m_g(n,f)$ for $g$ going from $0$ to a
user-defined maximum, for $n$ going from $2g$ to a user-defined maximum and from
$f$ going from $1$ to a maximum that makes $v = n - f + 2(1-g) = 1$ and
reinterprets $f$ as the number of vertices rather than the number of faces (by
face-vertex duality). He
conducted time trials on a portable Macintosh with a $2.66$ GHz Intel Core $2$
Duo processor, making $n$ range from $20$ to $100$ in steps of $10$ and setting
the maximum value of $g$ to be $n/2$. Here is the time in seconds for counting
rooted maps for each of these values of $n$ ($0$ means too small to be
measured).

\begin{center}
\begin{tabular}{r||r|r|r|r|r|r|r|r|r}
n    & 20 & 30 & 40 & 50 & 60 & 70 & 80 & 90 & 100
\\
\hline
time & 0  &  0 & 0.5 & 5 & 15 & 37 & 85 & 190 & 322
\end{tabular} 
\end{center}

We do not have a time-complexity analysis for counting unrooted maps, but in the
time trials it took less time to count unrooted maps than rooted ones. For $n =
100$, the time to count unrooted maps was $178$ seconds.

\section{Acknowledgments}
The authors wish to thank Valery Liskovets for bringing~\cite{CC14-v3} 
to their attention, and Alexander Mednykh for helpful comments.

\bibliographystyle{hplain}





\appendix

\newpage
\section*{Appendix: Numbers of unrooted maps}
\label{unrooted:appendix}
A table of numbers of unrooted planar maps (genus 0) with up to 6 edges
can be found in~\cite{Walsh83b}.  Larger tables of numbers of
 unrooted planar maps were computed by N.C. Wormald and given privately to the
 second author, but as far as we know, they have never been published.
A table of numbers of unrooted maps of genus 1-5 with up to 11 edges appears
in~\cite{W13} and~\cite{WGM12}.
We extend this table, both in terms of genus and number of edges, in this
appendix.

The following table gives the numbers $u_g(e,v)$ of genus-$g$ unrooted maps with
$e$ edges and $v$ vertices, for $g$ from $0$ to $19$ and for $v$ from $1$ to its
maximal value $e+1-2g$. The minimal value of $e$ is $2g$. The maximal value of
$e$ is arbitrarily fixed so that the table fits five pages for genera $0$ to
$2$, two pages for genera $3$ to $5$, and one page for higher genera.
The maximal value for $g$ is such that numbers fit in the page width.

\begin{longtable}{r|r|rrr}
$e$ & $v$ & $u_{0}(e,v)$ & $u_{1}(e,v)$ & $u_{2}(e,v)$ \\
\hline
\endhead
 & & & \\
0 & 1 & 1 & &  \\
 & & & \\
0 & sum  & 1  &  &  \\
 & & & \\
1 & 1 & 1 & &  \\
1 & 2 & 1 & &  \\
 & & & \\
1 & sum  & 2  &  &  \\
 & & & \\
2 & 1 & 1 & 1 &  \\
2 & 2 & 2 & &  \\
2 & 3 & 1 & &  \\
 & & & \\
2 & sum  & 4  & 1  &  \\
 & & & \\
3 & 1 & 2 & 3 &  \\
3 & 2 & 5 & 3 &  \\
3 & 3 & 5 & &  \\
3 & 4 & 2 & &  \\
 & & & \\
3 & sum  & 14  & 6  &  \\
 & & & \\
4 & 1 & 3 & 11 & 4  \\
4 & 2 & 14 & 24 &  \\
4 & 3 & 23 & 11 &  \\
4 & 4 & 14 & &  \\
4 & 5 & 3 & &  \\
 & & & \\
4 & sum  & 57  & 46  & 4  \\
 & & & \\
5 & 1 & 6 & 46 & 53  \\
5 & 2 & 42 & 180 & 53  \\
5 & 3 & 108 & 180 &  \\
5 & 4 & 108 & 46 &  \\
5 & 5 & 42 & &  \\
5 & 6 & 6 & &  \\
 & & & \\
5 & sum  & 312  & 452  & 106  \\
 & & & \\
6 & 1 & 14 & 204 & 553  \\
6 & 2 & 140 & 1198 & 1276  \\
6 & 3 & 501 & 2048 & 553  \\
6 & 4 & 761 & 1198 &  \\
6 & 5 & 501 & 204 &  \\
6 & 6 & 140 & &  \\
6 & 7 & 14 & &  \\
 & & & \\
6 & sum  & 2071  & 4852  & 2382  \\
 & & & \\
7 & 1 & 34 & 878 & 4758  \\
7 & 2 & 473 & 7212 & 18582  \\
7 & 3 & 2264 & 18396 & 18582  \\
7 & 4 & 4744 & 18396 & 4758  \\
7 & 5 & 4744 & 7212 &  \\
7 & 6 & 2264 & 878 &  \\
7 & 7 & 473 & &  \\
7 & 8 & 34 & &  \\
 & & & \\
7 & sum  & 15030  & 52972  & 46680  \\
 & & & \\
8 & 1 & 95 & 3799 & 35778  \\
8 & 2 & 1670 & 40776 & 205867  \\
8 & 3 & 10087 & 142727 & 347558  \\
8 & 4 & 27768 & 212443 & 205867  \\
8 & 5 & 38495 & 142727 & 35778  \\
8 & 6 & 27768 & 40776 &  \\
8 & 7 & 10087 & 3799 &  \\
8 & 8 & 1670 & &  \\
8 & 9 & 95 & &  \\
 & & & \\
8 & sum  & 117735  & 587047  & 830848  \\
 & & & \\
9 & 1 & 280 & 16304 & 244246  \\
9 & 2 & 5969 & 219520 & 1910756  \\
9 & 3 & 44310 & 999232 & 4747430  \\
9 & 4 & 153668 & 2040348 & 4747430  \\
9 & 5 & 279698 & 2040348 & 1910756  \\
9 & 6 & 279698 & 999232 & 244246  \\
9 & 7 & 153668 & 219520 &  \\
9 & 8 & 44310 & 16304 &  \\
9 & 9 & 5969 & &  \\
9 & 10 & 280 & &  \\
 & & & \\
9 & sum  & 967850  & 6550808  & 13804864  \\
 & & & \\
10 & 1 & 854 & 69486 & 1552834  \\
10 & 2 & 21679 & 1139075 & 15680071  \\
10 & 3 & 192444 & 6488604 & 52969260  \\
10 & 4 & 816661 & 17227356 & 77948670  \\
10 & 5 & 1873638 & 23634214 & 52969260  \\
10 & 6 & 2458264 & 17227356 & 15680071  \\
10 & 7 & 1873638 & 6488604 & 1552834  \\
10 & 8 & 816661 & 1139075 &  \\
10 & 9 & 192444 & 69486 &  \\
10 & 10 & 21679 & &  \\
10 & 11 & 854 & &  \\
 & & & \\
10 & sum  & 8268816  & 73483256  & 218353000  \\
 & & & \\
11 & 1 & 2694 & 294350 & 9349284  \\
11 & 2 & 79419 & 5741220 & 117450580  \\
11 & 3 & 828176 & 39779852 & 512308352  \\
11 & 4 & 4200980 & 132209016 & 1025303224  \\
11 & 5 & 11795964 & 235876296 & 1025303224  \\
11 & 6 & 19509632 & 235876296 & 512308352  \\
11 & 7 & 19509632 & 132209016 & 117450580  \\
11 & 8 & 11795964 & 39779852 & 9349284  \\
11 & 9 & 4200980 & 5741220 &  \\
11 & 10 & 828176 & 294350 &  \\
11 & 11 & 79419 & &  \\
11 & 12 & 2694 & &  \\
 & & & \\
11 & sum  & 72833730  & 827801468  & 3328822880  \\
 & & & \\
12 & 1 & 8714 & 1240308 & 53919954  \\
12 & 2 & 293496 & 28271474 & 819971501  \\
12 & 3 & 3537311 & 233068938 & 4452289504  \\
12 & 4 & 21061347 & 942568684 & 11509375296  \\
12 & 5 & 70719843 & 2105637162 & 15654660302  \\
12 & 6 & 143157616 & 2738550608 & 11509375296  \\
12 & 7 & 180492486 & 2105637162 & 4452289504  \\
12 & 8 & 143157616 & 942568684 & 819971501  \\
12 & 9 & 70719843 & 233068938 & 53919954  \\
12 & 10 & 21061347 & 28271474 &  \\
12 & 11 & 3537311 & 1240308 &  \\
12 & 12 & 293496 & &  \\
12 & 13 & 8714 & &  \\
 & & & \\
12 & sum  & 658049140  & 9360123740  & 49325772812  \\
 & & & \\
13 & 1 & 28640 & 5202148 & 300331878  \\
13 & 2 & 1091006 & 136580200 & 5412601192  \\
13 & 3 & 15014328 & 1316388936 & 35599161080  \\
13 & 4 & 103369288 & 6337310504 & 114602018272  \\
13 & 5 & 407569560 & 17232289072 & 201379328048  \\
13 & 6 & 986878680 & 28066908912 & 201379328048  \\
13 & 7 & 1523077528 & 28066908912 & 114602018272  \\
13 & 8 & 1523077528 & 17232289072 & 35599161080  \\
13 & 9 & 986878680 & 6337310504 & 5412601192  \\
13 & 10 & 407569560 & 1316388936 & 300331878  \\
13 & 11 & 103369288 & 136580200 &  \\
13 & 12 & 15014328 & 5202148 &  \\
13 & 13 & 1091006 & &  \\
13 & 14 & 28640 & &  \\
 & & & \\
13 & sum  & 6074058060  & 106189359544  & 714586880940  \\
 & & & \\
14 & 1 & 95640 & 21733696 & 1625426118  \\
14 & 2 & 4078213 & 649405334 & 34132653009  \\
14 & 3 & 63397256 & 7213525316 & 266220537080  \\
14 & 4 & 498495378 & 40620565952 & 1038541797978  \\
14 & 5 & 2273702888 & 131529397536 & 2273175492192  \\
14 & 6 & 6466334844 & 260810488496 & 2936946412728  \\
14 & 7 & 11939378311 & 326638072204 & 2273175492192  \\
14 & 8 & 14615468757 & 260810488496 & 1038541797978  \\
14 & 9 & 11939378311 & 131529397536 & 266220537080  \\
14 & 10 & 6466334844 & 40620565952 & 34132653009  \\
14 & 11 & 2273702888 & 7213525316 & 1625426118  \\
14 & 12 & 498495378 & 649405334 &  \\
14 & 13 & 63397256 & 21733696 &  \\
14 & 14 & 4078213 & &  \\
14 & 15 & 95640 & &  \\
 & & & \\
14 & sum  & 57106433817  & 1208328304864  & 10164338225482  \\
 & & & \\
15 & 1 & 323396 & 90493272 & 8587132844  \\
15 & 2 & 15312150 & 3046454992 & 207220225668  \\
15 & 3 & 266509050 & 38537828328 & 1884416656912  \\
15 & 4 & 2368459404 & 250230575696 & 8721265531848  \\
15 & 5 & 12343172450 & 948078314200 & 23138230175172  \\
15 & 6 & 40620147828 & 2239126384800 & 37241985748964  \\
15 & 7 & 88106500004 & 3414411073976 & 37241985748964  \\
15 & 8 & 129045594524 & 3414411073976 & 23138230175172  \\
15 & 9 & 129045594524 & 2239126384800 & 8721265531848  \\
15 & 10 & 88106500004 & 948078314200 & 1884416656912  \\
15 & 11 & 40620147828 & 250230575696 & 207220225668  \\
15 & 12 & 12343172450 & 38537828328 & 8587132844  \\
15 & 13 & 2368459404 & 3046454992 &  \\
15 & 14 & 266509050 & 90493272 &  \\
15 & 15 & 15312150 & &  \\
15 & 16 & 323396 & &  \\
 & & & \\
15 & sum  & 545532037612  & 13787042250528  & 142403410942816  \\
 & & & \\
16 & 1 & 1105335 & 375691885 & 44442582224  \\
16 & 2 & 57721030 & 14127535004 & 1218291353547  \\
16 & 3 & 1116113327 & 201485902915 & 12739188485210  \\
16 & 4 & 11110947214 & 1490633731778 & 68769605322980  \\
16 & 5 & 65472242053 & 6514453678793 & 216512936399236  \\
16 & 6 & 246254877247 & 18006841322290 & 422468300097440  \\
16 & 7 & 618198141193 & 32708686628027 & 526326450852812  \\
16 & 8 & 1063785332489 & 39826928417305 & 422468300097440  \\
16 & 9 & 1272842946261 & 32708686628027 & 216512936399236  \\
16 & 10 & 1063785332489 & 18006841322290 & 68769605322980  \\
16 & 11 & 618198141193 & 6514453678793 & 12739188485210  \\
16 & 12 & 246254877247 & 1490633731778 & 1218291353547  \\
16 & 13 & 65472242053 & 201485902915 & 44442582224  \\
16 & 14 & 11110947214 & 14127535004 &  \\
16 & 15 & 1116113327 & 375691885 &  \\
16 & 16 & 57721030 & &  \\
16 & 17 & 1105335 & &  \\
 & & & \\
16 & sum  & 5284835906037  & 157700137398689  & 1969831979334086  \\
 & & & \\
17 & 1 & 3813798 & 1555771800 & 225971343444  \\
17 & 2 & 218333832 & 64863745520 & 6968346176400  \\
17 & 3 & 4658894160 & 1033998837648 & 82820994884096  \\
17 & 4 & 51553861024 & 8628535594224 & 514298358102592  \\
17 & 5 & 340432303072 & 42979642352848 & 1889177369500464  \\
17 & 6 & 1448203830304 & 137064797207600 & 4375155072009488  \\
17 & 7 & 4155977725664 & 291433805486672 & 6608420098046976  \\
17 & 8 & 8278116804032 & 422739334207920 & 6608420098046976  \\
17 & 9 & 11637788525696 & 422739334207920 & 4375155072009488  \\
17 & 10 & 11637788525696 & 291433805486672 & 1889177369500464  \\
17 & 11 & 8278116804032 & 137064797207600 & 514298358102592  \\
17 & 12 & 4155977725664 & 42979642352848 & 82820994884096  \\
17 & 13 & 1448203830304 & 8628535594224 & 6968346176400  \\
17 & 14 & 340432303072 & 1033998837648 & 225971343444  \\
17 & 15 & 51553861024 & 64863745520 &  \\
17 & 16 & 4658894160 & 1555771800 &  \\
17 & 17 & 218333832 & &  \\
17 & 18 & 3813798 & &  \\
 & & & \\
17 & sum  & 51833908183164  & 1807893066408464  & 26954132420126920  \\
 & & & \\
18 & 1 & 13269146 & 6428291934 & 1131367963884  \\
18 & 2 & 828408842 & 295221527717 & 38919384594398  \\
18 & 3 & 19391786118 & 5221086204768 & 520644158094148  \\
18 & 4 & 236921843193 & 48720710849424 & 3676241660447931  \\
18 & 5 & 1739717050754 & 273824061235756 & 15538149312306360  \\
18 & 6 & 8295898355134 & 995529273862210 & 41992192647331392  \\
18 & 7 & 26931885143228 & 2442526267219360 & 75288406812106052  \\
18 & 8 & 61331742226722 & 4147624456667366 & 91282155067903038  \\
18 & 9 & 99813869859301 & 4941186214175258 & 75288406812106052  \\
18 & 10 & 117278995153034 & 4147624456667366 & 41992192647331392  \\
18 & 11 & 99813869859301 & 2442526267219360 & 15538149312306360  \\
18 & 12 & 61331742226722 & 995529273862210 & 3676241660447931  \\
18 & 13 & 26931885143228 & 273824061235756 & 520644158094148  \\
18 & 14 & 8295898355134 & 48720710849424 & 38919384594398  \\
18 & 15 & 1739717050754 & 5221086204768 & 1131367963884  \\
18 & 16 & 236921843193 & 295221527717 &  \\
18 & 17 & 19391786118 & 6428291934 &  \\
18 & 18 & 828408842 & &  \\
18 & 19 & 13269146 & &  \\
 & & & \\
18 & sum  & 514019531037910  & 20768681225892328  & 365393525753591368  \\
\end{longtable}

\newpage

\begin{longtable}{r|r|rrr}
$e$ & $v$ & $u_{3}(e,v)$ & $u_{4}(e,v)$ & $u_{5}(e,v)$ \\
\hline
\endhead
 & & & \\
6 & 1 & 131 & &  \\
 & & & \\
6 & sum  & 131  &  &  \\
 & & & \\
7 & 1 & 4079 & &  \\
7 & 2 & 4079 & &  \\
 & & & \\
7 & sum  & 8158  &  &  \\
 & & & \\
8 & 1 & 73282 & 14118 &  \\
8 & 2 & 167047 & &  \\
8 & 3 & 73282 & &  \\
 & & & \\
8 & sum  & 313611  & 14118  &  \\
 & & & \\
9 & 1 & 970398 & 684723 &  \\
9 & 2 & 3693031 & 684723 &  \\
9 & 3 & 3693031 & &  \\
9 & 4 & 970398 & &  \\
 & & & \\
9 & sum  & 9326858  & 1369446  &  \\
 & & & \\
10 & 1 & 10556722 & 17586433 & 2976853  \\
10 & 2 & 58591595 & 39630698 &  \\
10 & 3 & 97799324 & 17586433 &  \\
10 & 4 & 58591595 & &  \\
10 & 5 & 10556722 & &  \\
 & & & \\
10 & sum  & 236095958  & 74803564  & 2976853  \\
 & & & \\
11 & 1 & 99944546 & 319763792 & 195644427  \\
11 & 2 & 748976684 & 1192082898 & 195644427  \\
11 & 3 & 1823736772 & 1192082898 &  \\
11 & 4 & 1823736772 & 319763792 &  \\
11 & 5 & 748976684 & &  \\
11 & 6 & 99944546 & &  \\
 & & & \\
11 & sum  & 5345316004  & 3023693380  & 391288854  \\
 & & & \\
12 & 1 & 852737424 & 4631706389 & 6623379011  \\
12 & 2 & 8205279051 & 25016739573 & 14789629444  \\
12 & 3 & 26989340556 & 41395800249 & 6623379011  \\
12 & 4 & 39378084524 & 25016739573 &  \\
12 & 5 & 26989340556 & 4631706389 &  \\
12 & 6 & 8205279051 & &  \\
12 & 7 & 852737424 & &  \\
 & & & \\
12 & sum  & 111472798586  & 100692692173  & 28036387466  \\
 & & & \\
13 & 1 & 6709209232 & 56946090696 & 155182455738  \\
13 & 2 & 79996972480 & 413223640688 & 569441291708  \\
13 & 3 & 338043951088 & 991010148804 & 569441291708  \\
13 & 4 & 666422524608 & 991010148804 & 155182455738  \\
13 & 5 & 666422524608 & 413223640688 &  \\
13 & 6 & 338043951088 & 56946090696 &  \\
13 & 7 & 79996972480 & &  \\
13 & 8 & 6709209232 & &  \\
 & & & \\
13 & sum  & 2182345314816  & 2922359760376  & 1449247494892  \\
 & & & \\
14 & 1 & 49461969282 & 617936108012 & 2841197873030  \\
14 & 2 & 711640778177 & 5734881201032 & 15028479073373  \\
14 & 3 & 3728403936278 & 18485468237252 & 24701811831354  \\
14 & 4 & 9445619348392 & 26795029196244 & 15028479073373  \\
14 & 5 & 12763979300656 & 18485468237252 & 2841197873030  \\
14 & 6 & 9445619348392 & 5734881201032 &  \\
14 & 7 & 3728403936278 & 617936108012 &  \\
14 & 8 & 711640778177 & &  \\
14 & 9 & 49461969282 & &  \\
 & & & \\
14 & sum  & 40634231364914  & 76471600288836  & 60441165724160  \\
 & & & \\
15 & 1 & 345667110726 & 6074397541996 & 43425763829620  \\
15 & 2 & 5878587435378 & 69634518493584 & 307366103788730  \\
15 & 3 & 37184192378506 & 287198334481908 & 728458658338820  \\
15 & 4 & 116833971177188 & 559637322350992 & 728458658338820  \\
15 & 5 & 202918633990626 & 559637322350992 & 307366103788730  \\
15 & 6 & 202918633990626 & 287198334481908 & 43425763829620  \\
15 & 7 & 116833971177188 & 69634518493584 &  \\
15 & 8 & 37184192378506 & 6074397541996 &  \\
15 & 9 & 5878587435378 & &  \\
15 & 10 & 345667110726 & &  \\
 & & & \\
15 & sum  & 726322104184848  & 1845089145736960  & 2158501051914340  \\
 & & & \\
16 & 1 & 2310028835346 & 55099526091224 & 577374933310906  \\
16 & 2 & 45675916449962 & 760174730620316 & 5209797503498640  \\
16 & 3 & 341686270713324 & 3874407685623078 & 16527742407430762  \\
16 & 4 & 1296601404482135 & 9662433645931070 & 23833896316372268  \\
16 & 5 & 2793465994063884 & 12990353144165406 & 16527742407430762  \\
16 & 6 & 3590596058829058 & 9662433645931070 & 5209797503498640  \\
16 & 7 & 2793465994063884 & 3874407685623078 & 577374933310906  \\
16 & 8 & 1296601404482135 & 760174730620316 &  \\
16 & 9 & 341686270713324 & 55099526091224 &  \\
16 & 10 & 45675916449962 & &  \\
16 & 11 & 2310028835346 & &  \\
 & & & \\
16 & sum  & 12550075287918360  & 41694584320696782  & 68463726004852884  \\
\end{longtable}

\newpage

\begin{longtable}{r|r|rr}
$e$ & $v$ & $u_{6}(e,v)$ & $u_{7}(e,v)$ \\
\hline
\endhead
 & & \\
12 & 1 & 1013582110 &  \\
 & & \\
12 & sum  & 1013582110  &  \\
 & & \\
13 & 1 & 84928729933 &  \\
13 & 2 & 84928729933 &  \\
 & & \\
13 & sum  & 169857459866  &  \\
 & & \\
14 & 1 & 3605028726801 & 508233789579  \\
14 & 2 & 7992502487664 &  \\
14 & 3 & 3605028726801 &  \\
 & & \\
14 & sum  & 15202559941266  & 508233789579  \\
 & & \\
15 & 1 & 104340300511680 & 52147993673063  \\
15 & 2 & 378134298777037 & 52147993673063  \\
15 & 3 & 378134298777037 &  \\
15 & 4 & 104340300511680 &  \\
 & & \\
15 & sum  & 964949198577434  & 104295987346126  \\
 & & \\
16 & 1 & 2328771846722608 & 2680480846764174  \\
16 & 2 & 12115285934958463 & 5908630695597199  \\
16 & 3 & 19807266932574138 & 2680480846764174  \\
16 & 4 & 12115285934958463 &  \\
16 & 5 & 2328771846722608 &  \\
 & & \\
16 & sum  & 48695382495936280  & 11269592389125547  \\
 & & \\
17 & 1 & 42879330119010060 & 92968027407241048  \\
17 & 2 & 297515608385017712 & 333529278137138064  \\
17 & 3 & 698425724113143808 & 333529278137138064  \\
17 & 4 & 698425724113143808 & 92968027407241048  \\
17 & 5 & 297515608385017712 &  \\
17 & 6 & 42879330119010060 &  \\
 & & \\
17 & sum  & 2077641325234343160  & 852994611088758224  \\
 & & \\
18 & 1 & 679574571686566150 & 2462686849706956592  \\
18 & 2 & 5994737178116108922 & 12639396448986592872  \\
18 & 3 & 18774176953946323287 & 20573712843056206498  \\
18 & 4 & 26959977164375096074 & 12639396448986592872  \\
18 & 5 & 18774176953946323287 & 2462686849706956592  \\
18 & 6 & 5994737178116108922 &  \\
18 & 7 & 679574571686566150 &  \\
 & & \\
18 & sum  & 77856954571873092792  & 50777879440443305426  \\
\end{longtable}

\newpage

\begin{longtable}{r|r|rr}
$e$ & $v$ & $u_{8}(e,v)$ & $u_{9}(e,v)$ \\
\hline
\endhead
 & & \\
16 & 1 & 352755124921122 &  \\
 & & \\
16 & sum  & 352755124921122  &  \\
 & & \\
17 & 1 & 43058443920636593 &  \\
17 & 2 & 43058443920636593 &  \\
 & & \\
17 & sum  & 86116887841273186  &  \\
 & & \\
18 & 1 & 2612103505736970587 & 324039613564554401  \\
18 & 2 & 5730580864933991642 &  \\
18 & 3 & 2612103505736970587 &  \\
 & & \\
18 & sum  & 10954787876407932816  & 324039613564554401  \\
 & & \\
19 & 1 & 106104636805432131380 & 46037869184438374355  \\
19 & 2 & 377468533878532051274 & 46037869184438374355  \\
19 & 3 & 377468533878532051274 &  \\
19 & 4 & 106104636805432131380 &  \\
 & & \\
19 & sum  & 967146341367928365308  & 92075738368876748710  \\
 & & \\
20 & 1 & 3268090017604446925695 & 3231706843486368031963  \\
20 & 2 & 16583906258119918465914 & 7061507183694710755564  \\
20 & 3 & 26895334324381935980135 & 3231706843486368031963  \\
20 & 4 & 16583906258119918465914 &  \\
20 & 5 & 3268090017604446925695 &  \\
 & & \\
20 & sum  & 66599326875830666763353  & 13524920870667446819490  \\
 & & \\
21 & 1 & 81763508749452267702334 & 151020126911739994806940  \\
21 & 2 & 550484901682834216964372 & 533436706524721228557255  \\
21 & 3 & 1273683419173516774041758 & 533436706524721228557255  \\
21 & 4 & 1273683419173516774041758 & 151020126911739994806940  \\
21 & 5 & 550484901682834216964372 &  \\
21 & 6 & 81763508749452267702334 &  \\
 & & \\
21 & sum  & 3811863659211606517416928  & 1368913666872922446728390  \\
 & & \\
22 & 1 & 1735799012483201542629310 & 5321407675084935890385252  \\
22 & 2 & 14793365548485479622939589 & 26743956334292711312949466  \\
22 & 3 & 45422458688847126828542788 & 43236837587662714557906902  \\
22 & 4 & 64807620764474890716233843 & 26743956334292711312949466  \\
22 & 5 & 45422458688847126828542788 & 5321407675084935890385252  \\
22 & 6 & 14793365548485479622939589 &  \\
22 & 7 & 1735799012483201542629310 &  \\
 & & \\
22 & sum  & 188710867264106506704457217  & 107367565606418008964576338  \\
\end{longtable}

\newpage

\begin{longtable}{r|r|rr}
$e$ & $v$ & $u_{10}(e,v)$ & $u_{11}(e,v)$ \\
\hline
\endhead
 & & \\
20 & 1 & 380751174738424280720 &  \\
 & & \\
20 & sum  & 380751174738424280720  &  \\
 & & \\
21 & 1 & 61900350644739074439445 &  \\
21 & 2 & 61900350644739074439445 &  \\
 & & \\
21 & sum  & 123800701289478148878890  &  \\
 & & \\
22 & 1 & 4950082376594691225742201 & 557175918657122229139987  \\
22 & 2 & 10779106107210130980277396 &  \\
22 & 3 & 4950082376594691225742201 &  \\
 & & \\
22 & sum  & 20679270860399513431761798  & 557175918657122229139987  \\
 & & \\
23 & 1 & 262334467319926285470894622 & 102246856493968374607463423  \\
23 & 2 & 920946297304483463377091334 & 102246856493968374607463423  \\
23 & 3 & 920946297304483463377091334 &  \\
23 & 4 & 262334467319926285470894622 &  \\
 & & \\
23 & sum  & 2366561529248819497695971912  & 204493712987936749214926846  \\
 & & \\
24 & 1 & 10436807330236403989810496394 & 9197643937243060077009264642  \\
24 & 2 & 52016489918140456428526385133 & 19968543728518640843922922089  \\
24 & 3 & 83866835662326132082607669996 & 9197643937243060077009264642  \\
24 & 4 & 52016489918140456428526385133 &  \\
24 & 5 & 10436807330236403989810496394 &  \\
 & & \\
24 & sum  & 208773430159079852919281433050  & 38363831603004760997941451373  \\
 & & \\
25 & 1 & 334074519898673454431872772378 & 546349421734820701894788862980  \\
25 & 2 & 2200395535965333545033744718386 & 1907742061916852507697868346104  \\
25 & 3 & 5037409011844160996918020411320 & 1907742061916852507697868346104  \\
25 & 4 & 5037409011844160996918020411320 & 546349421734820701894788862980  \\
25 & 5 & 2200395535965333545033744718386 &  \\
25 & 6 & 334074519898673454431872772378 &  \\
 & & \\
25 & sum  & 15143758135416335992767275804168  & 4908182967303346419185314418168  \\
 & & \\
26 & 1 & 8992412804931496094769804314194 & 24276552615926015429243306726942  \\
26 & 2 & 74724657022260381172998180758989 & 120111902847763968111649111952181  \\
26 & 3 & 226115216635966966996943064786668 & 193197660432676247606899872716724  \\
26 & 4 & 321078006189133085356572309019684 & 120111902847763968111649111952181  \\
26 & 5 & 226115216635966966996943064786668 & 24276552615926015429243306726942  \\
26 & 6 & 74724657022260381172998180758989 &  \\
26 & 7 & 8992412804931496094769804314194 &  \\
 & & \\
26 & sum  & 940742579115450773885994408739386  & 481974571360056214688684710074970  \\
\end{longtable}

\newpage

\begin{longtable}{r|r|r}
$e$ & $v$ & $u_{12}(e,v)$ \\
\hline
\endhead
 & \\
24 & 1 & 993806827312044893602464496  \\
 & \\
24 & sum  & 993806827312044893602464496  \\
 & \\
25 & 1 & 203568251472192593015565105153  \\
25 & 2 & 203568251472192593015565105153  \\
 & \\
25 & sum  & 407136502944385186031130210306  \\
 & \\
26 & 1 & 20384681425578629630065436540001  \\
26 & 2 & 44139689150396597299844837950650  \\
26 & 3 & 20384681425578629630065436540001  \\
 & \\
26 & sum  & 84909052001553856559975711030652  \\
 & \\
27 & 1 & 1344032802282918564446470093660642  \\
27 & 2 & 4670923997634519162591788341200373  \\
27 & 3 & 4670923997634519162591788341200373  \\
27 & 4 & 1344032802282918564446470093660642  \\
 & \\
27 & sum  & 12029913599834875454076516869722030  \\
 & \\
28 & 1 & 66095585390798198366295835295463407  \\
28 & 2 & 324909698939260325755902147746674575  \\
28 & 3 & 521510766921112908280233047150555771  \\
28 & 4 & 324909698939260325755902147746674575  \\
28 & 5 & 66095585390798198366295835295463407  \\
 & \\
28 & sum  & 1303521335581229956524629013234831735  \\
 & \\
29 & 1 & 2598647372085586327745995717592768700  \\
29 & 2 & 16829619487924984024663810973279277072  \\
29 & 3 & 38214110448983922489609206857083746964  \\
29 & 4 & 38214110448983922489609206857083746964  \\
29 & 5 & 16829619487924984024663810973279277072  \\
29 & 6 & 2598647372085586327745995717592768700  \\
 & \\
29 & sum  & 115284754617988985684038027095911585472  \\
 & \\
30 & 1 & 85392758017801624687683117609052737636  \\
30 & 2 & 695858144741566474009926149792630696265  \\
30 & 3 & 2081944949854120175333709476140932999277  \\
30 & 4 & 2945354467291151637150567124347940326546  \\
30 & 5 & 2081944949854120175333709476140932999277  \\
30 & 6 & 695858144741566474009926149792630696265  \\
30 & 7 & 85392758017801624687683117609052737636  \\
 & \\
30 & sum  & 8671746172518128185213204611433173192902  \\
\end{longtable}

\newpage

\begin{longtable}{r|r|r}
$e$ & $v$ & $u_{13}(e,v)$ \\
\hline
\endhead
 & \\
26 & 1 & 2122669454233128302149617542253  \\
 & \\
26 & sum  & 2122669454233128302149617542253  \\
 & \\
27 & 1 & 480832429153352742558421356793665  \\
27 & 2 & 480832429153352742558421356793665  \\
 & \\
27 & sum  & 961664858306705485116842713587330  \\
 & \\
28 & 1 & 53130366325356854395291828009938727  \\
28 & 2 & 114775986070991484071278106575375896  \\
28 & 3 & 53130366325356854395291828009938727  \\
 & \\
28 & sum  & 221036718721705192861861762595253350  \\
 & \\
29 & 1 & 3856356911137003381320301936446461074  \\
29 & 2 & 13345600579548980014352553597690495308  \\
29 & 3 & 13345600579548980014352553597690495308  \\
29 & 4 & 3856356911137003381320301936446461074  \\
 & \\
29 & sum  & 34403914981371966791345711068273912764  \\
 & \\
30 & 1 & 208265988956285322041288424854784582594  \\
30 & 2 & 1017889232438589381522421291029823820959  \\
30 & 3 & 1630734704264982296160323672383390044428  \\
30 & 4 & 1017889232438589381522421291029823820959  \\
30 & 5 & 208265988956285322041288424854784582594  \\
 & \\
30 & sum  & 4083045147054731703287743104152606851534  \\
 & \\
31 & 1 & 8970722042074945753967336008761272150320  \\
31 & 2 & 57686177099045026849910147518610059988352  \\
31 & 3 & 130533817193066912073725928877428968894744  \\
31 & 4 & 130533817193066912073725928877428968894744  \\
31 & 5 & 57686177099045026849910147518610059988352  \\
31 & 6 & 8970722042074945753967336008761272150320  \\
 & \\
31 & sum  & 394381432668373769355206824809600602066832  \\
 & \\
32 & 1 & 322189870761730650250511467871042510892076  \\
32 & 2 & 2603945884819023919715984436445634621702948  \\
32 & 3 & 7753587725454287210291908773219679038077036  \\
32 & 4 & 10951890382776236471789968498330693093538848  \\
32 & 5 & 7753587725454287210291908773219679038077036  \\
32 & 6 & 2603945884819023919715984436445634621702948  \\
32 & 7 & 322189870761730650250511467871042510892076  \\
 & \\
32 & sum  & 32311337344846320032306777853403405434882968  \\
\end{longtable}

\newpage

\begin{longtable}{r|r|r}
$e$ & $v$ & $u_{14}(e,v)$ \\
\hline
\endhead
 & \\
28 & 1 & 5349362295912408418285480950292454  \\
 & \\
28 & sum  & 5349362295912408418285480950292454  \\
 & \\
29 & 1 & 1329513645388215594553239451794715965  \\
29 & 2 & 1329513645388215594553239451794715965  \\
 & \\
29 & sum  & 2659027290776431189106478903589431930  \\
 & \\
30 & 1 & 160898561634069521892363237502437777425  \\
30 & 2 & 346855611453502747893618316325861938456  \\
30 & 3 & 160898561634069521892363237502437777425  \\
 & \\
30 & sum  & 668652734721641791678344791330737493306  \\
 & \\
31 & 1 & 12765769708416431289559831195493899470210  \\
31 & 2 & 44011104918179662694793003075316613282622  \\
31 & 3 & 44011104918179662694793003075316613282622  \\
31 & 4 & 12765769708416431289559831195493899470210  \\
 & \\
31 & sum  & 113553749253192187968705668541621025505664  \\
 & \\
32 & 1 & 752086797680822688120006130623923595622946  \\
32 & 2 & 3656730523358976901036890991989391463234050  \\
32 & 3 & 5848430055307490353596054635418558343551050  \\
32 & 4 & 3656730523358976901036890991989391463234050  \\
32 & 5 & 752086797680822688120006130623923595622946  \\
 & \\
32 & sum  & 14666064697387089531909848880645188461265042  \\
 & \\
33 & 1 & 35267296779456996454041460582407284988489428  \\
33 & 2 & 225339610956320336710389271871310435524097264  \\
33 & 3 & 508319887252694286639071826664484584536471066  \\
33 & 4 & 508319887252694286639071826664484584536471066  \\
33 & 5 & 225339610956320336710389271871310435524097264  \\
33 & 6 & 35267296779456996454041460582407284988489428  \\
 & \\
33 & sum  & 1537853589976943239607005118236404610098115516  \\
 & \\
34 & 1 & 1376188688409635874905853343655201243854733608  \\
34 & 2 & 11039921104104953238907248969596220332376818700  \\
34 & 3 & 32730965901080739246379167442830648082201463568  \\
34 & 4 & 46166673137751515758000677859503539538189538108  \\
34 & 5 & 32730965901080739246379167442830648082201463568  \\
34 & 6 & 11039921104104953238907248969596220332376818700  \\
34 & 7 & 1376188688409635874905853343655201243854733608  \\
 & \\
34 & sum  & 136460824524942172478385217371667678855055569860  \\
\end{longtable}

\newpage

\begin{longtable}{r|r|r}
$e$ & $v$ & $u_{15}(e,v)$ \\
\hline
\endhead
 & \\
30 & 1 & 15707315253480198543039354159336702543  \\
 & \\
30 & sum  & 15707315253480198543039354159336702543  \\
 & \\
31 & 1 & 4254404066846916348883588403716819128103  \\
31 & 2 & 4254404066846916348883588403716819128103  \\
 & \\
31 & sum  & 8508808133693832697767176807433638256206  \\
 & \\
32 & 1 & 560291242973155478934912238951878579485370  \\
32 & 2 & 1205555508775085796667646968133497238927983  \\
32 & 3 & 560291242973155478934912238951878579485370  \\
 & \\
32 & sum  & 2326137994721396754537471446037254397898723  \\
 & \\
33 & 1 & 48296285125745000564523254650901622578745046  \\
33 & 2 & 165935861505220722116579262978988782927590263  \\
33 & 3 & 165935861505220722116579262978988782927590263  \\
33 & 4 & 48296285125745000564523254650901622578745046  \\
 & \\
33 & sum  & 428464293261931445362205035259780811012670618  \\
 & \\
34 & 1 & 3085985802474307761706438171133166671774221272  \\
34 & 2 & 14933910157176516402233695721551256357086244372  \\
34 & 3 & 23848012464697728367174976612736644751185622064  \\
34 & 4 & 14933910157176516402233695721551256357086244372  \\
34 & 5 & 3085985802474307761706438171133166671774221272  \\
 & \\
34 & sum  & 59887804383999376695055244398105490808906553352  \\
 & \\
35 & 1 & 156675496336534578862661113608020513244779757004  \\
35 & 2 & 995278045846569420039029893381074747025407746934  \\
35 & 3 & 2238805803235545167737076841479360428504150689292  \\
35 & 4 & 2238805803235545167737076841479360428504150689292  \\
35 & 5 & 995278045846569420039029893381074747025407746934  \\
35 & 6 & 156675496336534578862661113608020513244779757004  \\
 & \\
35 & sum  & 6781518690837298333277535696936911377548676386460  \\
 & \\
36 & 1 & 6607804650427895027651495257432388201136957637604  \\
36 & 2 & 52651158391352383107127670390450702272652846669207  \\
36 & 3 & 155486533895708212045879189159875151451391021408966  \\
36 & 4 & 219028932036157434744940729921104818730252539947224  \\
36 & 5 & 155486533895708212045879189159875151451391021408966  \\
36 & 6 & 52651158391352383107127670390450702272652846669207  \\
36 & 7 & 6607804650427895027651495257432388201136957637604  \\
 & \\
36 & sum  & 648519925911134415106257439536621302580614191378778  \\
\end{longtable}

\newpage

\begin{longtable}{r|r|r}
$e$ & $v$ & $u_{16}(e,v)$ \\
\hline
\endhead
 & \\
32 & 1 & 53160783752637968542926390100726167551610  \\
 & \\
32 & sum  & 53160783752637968542926390100726167551610  \\
 & \\
33 & 1 & 15600320175507300729043193735546421507235073  \\
33 & 2 & 15600320175507300729043193735546421507235073  \\
 & \\
33 & sum  & 31200640351014601458086387471092843014470146  \\
 & \\
34 & 1 & 2223274627983595422507340640133497482768944961  \\
34 & 2 & 4775494679658728042571008949941678370459941602  \\
34 & 3 & 2223274627983595422507340640133497482768944961  \\
 & \\
34 & sum  & 9222043935625918887585690230208673335997831524  \\
 & \\
35 & 1 & 207098988317023096921058787758993781463514981524  \\
35 & 2 & 709331542316073084891616565646314491079275188466  \\
35 & 3 & 709331542316073084891616565646314491079275188466  \\
35 & 4 & 207098988317023096921058787758993781463514981524  \\
 & \\
35 & sum  & 1832861061266192363625350706810616545085580339980  \\
 & \\
36 & 1 & 14279273921214976669647789455348595104544279637003  \\
36 & 2 & 68805165020043702416786649792601959818097931393304  \\
36 & 3 & 109721245211691855532004451836827883534451504403689  \\
36 & 4 & 68805165020043702416786649792601959818097931393304  \\
36 & 5 & 14279273921214976669647789455348595104544279637003  \\
 & \\
36 & sum  & 275890123094209213704873330332728993379735926464303  \\
 & \\
37 & 1 & 781105068702241754910796845876629366053531576061370  \\
37 & 2 & 4935719985964953802862159423860034674318681677247312  \\
37 & 3 & 11073916808061482567461763466695337023784694521589900  \\
37 & 4 & 11073916808061482567461763466695337023784694521589900  \\
37 & 5 & 4935719985964953802862159423860034674318681677247312  \\
37 & 6 & 781105068702241754910796845876629366053531576061370  \\
 & \\
37 & sum  & 33581483725457356250469439472864002128313815549797164  \\
 & \\
38 & 1 & 35441207898429792043123436123187064049015449552488250  \\
38 & 2 & 280656201598016620390674260860638474559649068763638975  \\
38 & 3 & 825842658906095001495942350546153309388786108632837332  \\
38 & 4 & 1161964050508679879461554362247938069968931276241633907  \\
38 & 5 & 825842658906095001495942350546153309388786108632837332  \\
38 & 6 & 280656201598016620390674260860638474559649068763638975  \\
38 & 7 & 35441207898429792043123436123187064049015449552488250  \\
 & \\
38 & sum  & 3445844187313762707321034457307895765963832530139563021  \\
\end{longtable}

\newpage

\begin{longtable}{r|r|r}
$e$ & $v$ & $u_{17}(e,v)$ \\
\hline
\endhead
 & \\
34 & 1 & 205445432928009832581491069006516880609705841  \\
 & \\
34 & sum  & 205445432928009832581491069006516880609705841  \\
 & \\
35 & 1 & 64986432002837812745875711198508285517101944963  \\
35 & 2 & 64986432002837812745875711198508285517101944963  \\
 & \\
35 & sum  & 129972864005675625491751422397016571034203889926  \\
 & \\
36 & 1 & 9973161581095456726301922365556925358027314253479  \\
36 & 2 & 21388203433438836924480615057329648012547524988820  \\
36 & 3 & 9973161581095456726301922365556925358027314253479  \\
 & \\
36 & sum  & 41334526595629750377084459788443498728602153495778  \\
 & \\
37 & 1 & 999214401417939130045637235023670174633866987201502  \\
37 & 2 & 3412631887102397881486610278438413589213538056913876  \\
37 & 3 & 3412631887102397881486610278438413589213538056913876  \\
37 & 4 & 999214401417939130045637235023670174633866987201502  \\
 & \\
37 & sum  & 8823692577040674023064495026924167527694810088230756  \\
 & \\
38 & 1 & 74008165218025656228403282401053203578535203947902856  \\
38 & 2 & 355210215550916519986577506055980731891976679082990264  \\
38 & 3 & 565713794642371539921830564202098813295282815677680430  \\
38 & 4 & 355210215550916519986577506055980731891976679082990264  \\
38 & 5 & 74008165218025656228403282401053203578535203947902856  \\
 & \\
38 & sum  & 1424150556180255892351792141116166684236306581739466670  \\
 & \\
39 & 1 & 4343202278169258593275684010568229749944220951468232434  \\
39 & 2 & 27311052025745996059683391861883339932687540212345629528  \\
39 & 3 & 61130753148022353058797706674786842441835799506625894124  \\
39 & 4 & 61130753148022353058797706674786842441835799506625894124  \\
39 & 5 & 27311052025745996059683391861883339932687540212345629528  \\
39 & 6 & 4343202278169258593275684010568229749944220951468232434  \\
 & \\
39 & sum  & 185570014903875215423513565094476824248935121340879512172  \\
 & \\
40 & 1 & 211137508713213986158094387442923550756912980107062330239  \\
40 & 2 & 1662514046685373766481849147320976176363162317404644178392  \\
40 & 3 & 4875869143426687082059537470267033644594639078062839477983  \\
40 & 4 & 6852916545721885675371996110251583758639955478551923500331  \\
40 & 5 & 4875869143426687082059537470267033644594639078062839477983  \\
40 & 6 & 1662514046685373766481849147320976176363162317404644178392  \\
40 & 7 & 211137508713213986158094387442923550756912980107062330239  \\
 & \\
40 & sum  & 20351957943372435344770958120313450502069384229701015473559  \\
\end{longtable}

\newpage

\begin{longtable}{r|r|r}
$e$ & $v$ & $u_{18}(e,v)$ \\
\hline
\endhead
 & \\
36 & 1 & 899180059563093845406786676951930933700666538428  \\
 & \\
36 & sum  & 899180059563093845406786676951930933700666538428  \\
 & \\
37 & 1 & 305209945591860728601292564295170264716895126366229  \\
37 & 2 & 305209945591860728601292564295170264716895126366229  \\
 & \\
37 & sum  & 610419891183721457202585128590340529433790252732458  \\
 & \\
38 & 1 & 50219426896416777538218063589516219785265760267129801  \\
38 & 2 & 107543707445670286349292801498822861069642312371491692  \\
38 & 3 & 50219426896416777538218063589516219785265760267129801  \\
 & \\
38 & sum  & 207982561238503841425728928677855300640173832905751294  \\
 & \\
39 & 1 & 5389202115814743871428660271901529819121489449229108268  \\
39 & 2 & 18357522907485928372236446971581267536474065605106694205  \\
39 & 3 & 18357522907485928372236446971581267536474065605106694205  \\
39 & 4 & 5389202115814743871428660271901529819121489449229108268  \\
 & \\
39 & sum  & 47493450046601344487330214486965594711191110108671604946  \\
 & \\
40 & 1 & 427068485931184787066302658812981942725804369479622658290  \\
40 & 2 & 2042346003975213807545249105132487353129138259822245439027  \\
40 & 3 & 3248824444658530733670030530464329613423303443228840360174  \\
40 & 4 & 2042346003975213807545249105132487353129138259822245439027  \\
40 & 5 & 427068485931184787066302658812981942725804369479622658290  \\
 & \\
40 & sum  & 8187653424471327922893134058355268205133188701832576554808  \\
 & \\
41 & 1 & 26784645725041818334522973443000470040238838410719829307538  \\
41 & 2 & 167674309375247041644579797461104409791966365092985054355672  \\
41 & 3 & 374488402865205525083259282243771736446866572624028195496896  \\
41 & 4 & 374488402865205525083259282243771736446866572624028195496896  \\
41 & 5 & 167674309375247041644579797461104409791966365092985054355672  \\
41 & 6 & 26784645725041818334522973443000470040238838410719829307538  \\
 & \\
41 & sum  & 1137894715930988770124724106295753232558143552255466158320212  \\
 & \\
42 & 1 & 1389944269863838497811621155190654489220344748250650534194358  \\
42 & 2 & 10887331081526501822853177759545887968080128197675407799362682  \\
42 & 3 & 31833348001308959952129406291954071592650439711678587040169837  \\
42 & 4 & 44696120539232218172253544414937593313423676099859800982308502  \\
42 & 5 & 31833348001308959952129406291954071592650439711678587040169837  \\
42 & 6 & 10887331081526501822853177759545887968080128197675407799362682  \\
42 & 7 & 1389944269863838497811621155190654489220344748250650534194358  \\
 & \\
42 & sum  & 132917367244630818717841954828318821413325501415069091729762256  \\
\end{longtable}

\newpage

\begin{longtable}{r|r|r}
$e$ & $v$ & $u_{19}(e,v)$ \\
\hline
\endhead
 & \\
38 & 1 & 4424730378121305321456186121529010463964830910484003  \\
 & \\
38 & sum  & 4424730378121305321456186121529010463964830910484003  \\
 & \\
39 & 1 & 1605186831344690925467081160430702520300496445763873311  \\
39 & 2 & 1605186831344690925467081160430702520300496445763873311  \\
 & \\
39 & sum  & 3210373662689381850934162320861405040600992891527746622  \\
 & \\
40 & 1 & 282085119538112569201481233917466805746825156704653052178  \\
40 & 2 & 603272708788070683140741539446707320754385893163661763443  \\
40 & 3 & 282085119538112569201481233917466805746825156704653052178  \\
 & \\
40 & sum  & 1167442947864295821543704007281640932248036206572967867799  \\
 & \\
41 & 1 & 32302728422532346841074282111005042311763117883791879554812  \\
41 & 2 & 109767667148746367413775667829043306525555960548263837730600  \\
41 & 3 & 109767667148746367413775667829043306525555960548263837730600  \\
41 & 4 & 32302728422532346841074282111005042311763117883791879554812  \\
 & \\
41 & sum  & 284140791142557428509699899880096697674638156864111434570824  \\
 & \\
42 & 1 & 2729008813037064165588608037007491955895014418138507486290090  \\
42 & 2 & 13007112201580570259600784438565181934710759194453710021264759  \\
42 & 3 & 20668163895282079484958851224259658246014568460489861510080658  \\
42 & 4 & 13007112201580570259600784438565181934710759194453710021264759  \\
42 & 5 & 2729008813037064165588608037007491955895014418138507486290090  \\
 & \\
42 & sum  & 52140405924517348335337636175405006027226115685674296525190356  \\
 & \\
43 & 1 & 182285420872810479366334700885431830095294173825146404386445914  \\
43 & 2 & 1136394999010390323212649937280709365011343445435404008508674492  \\
43 & 3 & 2532921941159641534900117827607810131249613553947035703033638132  \\
43 & 4 & 2532921941159641534900117827607810131249613553947035703033638132  \\
43 & 5 & 1136394999010390323212649937280709365011343445435404008508674492  \\
43 & 6 & 182285420872810479366334700885431830095294173825146404386445914  \\
 & \\
43 & sum  & 7703204722085684674958204931547902652712502346415172231857517076  \\
 & \\
44 & 1 & 10064181576960495335927932858268075358326289452842097820366680160  \\
44 & 2 & 78450121794974638316938365981758852583573376009308176509400884718  \\
44 & 3 & 228731201293361297266257064521272334489573226866484237764617340414  \\
44 & 4 & 320854956051603947387522506512966561014063478810265368757633296296  \\
44 & 5 & 228731201293361297266257064521272334489573226866484237764617340414  \\
44 & 6 & 78450121794974638316938365981758852583573376009308176509400884718  \\
44 & 7 & 10064181576960495335927932858268075358326289452842097820366680160  \\
 & \\
44 & sum  & 955345965382196809225769233235565085877009263467534392946403106880  \\
\end{longtable}


\begin{thebibliography}{10}

\bibitem{Arq87a}
D.~Arqu{\`e}s.
\newblock Relations fonctionnelles et d{\'e}nombrement des cartes point{\'e}es
  sur le tore.
\newblock {\em J. Combin. Theory Ser. B}, 43(3):253--274, 1987.

\bibitem{ag99}
D.~Arqu\`es and A.~Giorgetti.
\newblock {\'E}num\'eration des cartes point\'ees de genre quelconque en
  fonction des nombres de sommets et de faces.
\newblock {\em J. Combin. Theory Ser. B}, 77(1):1--24, 1999.

\bibitem{BC86}
E.~A. Bender and E.~R. Canfield.
\newblock The asymptotic number of rooted maps on a surface.
\newblock {\em J. Combin. Theory Ser. A}, 43(2):244--257, 1986.

\bibitem{BC91}
E.~A. Bender and E.~R. Canfield.
\newblock The number of rooted maps on an orientable surface.
\newblock {\em J. Combin. Theory Ser. B}, 53(2):293--299, 1991.

\bibitem{BCR93}
E.~A. Bender, E.~R. Canfield, and L.~B. Richmond.
\newblock The asymptotic number of rooted maps on a surface. ii. enumeration by
  vertices and faces.
\newblock {\em J. Comb. Theory, Ser. A}, 63(2):318--329, 1993.

\bibitem{CC14-v3}
S.~R. Carrell and G.~Chapuy.
\newblock Simple recurrence formulas to count maps on orientable surfaces.
\newblock {\em ArXiv e-prints}, 2014, 1402.6300.
\newblock Submitted on 19 March. URL \url{http://arxiv.org/abs/1402.6300}.

\bibitem{gio98a}
A.~Giorgetti.
\newblock {\em Combinatoire bijective et \'enum\'erative des cartes point\'ees
  sur une surface}.
\newblock PhD thesis, Universit\'e de Marne-la-Vall\'ee, Institut Gaspard
  Monge, 1998.

\bibitem{giorgetti10:ip}
A.~Giorgetti.
\newblock {Guessing a Conjecture in Enumerative Combinatorics and Proving It
  with a Computer Algebra System}.
\newblock In T.~Jebelean, M.~Mosbah, and N.~Popov, editors, {\em {SCSS'10}},
  pages 5--18, July 2010.

\bibitem{MAP}
A.~Giorgetti.
\newblock {MAP project}, 2014.
\newblock \url{https://sourceforge.net/projects/combi/files/map}.

\bibitem{Liskovets10}
V.A. Liskovets.
\newblock A multivariate arithmetic function of combinatorial and topological
  significance.
\newblock {\em Integers}, 10:155--177, 2010.

\bibitem{MG11}
A.~Mednykh and A.~Giorgetti.
\newblock Enumeration of genus four maps by number of edges.
\newblock {\em Ars Mathematica Contemporanea}, 4:351--361, 2011.

\bibitem{MN06}
A.~Mednykh and R.~Nedela.
\newblock Enumeration of unrooted maps of a given genus.
\newblock {\em J. Comb. Theory, Ser. B}, 96(5):706--729, 2006.

\bibitem{SS71}
A.~Schönhage and V.~Strassen.
\newblock {Schnelle Multiplikation großer Zahlen}.
\newblock {\em Computing}, 7(3-4):281--292, 1971.

\bibitem{Stanley2}
R.~P. Stanley.
\newblock {\em Enumerative Combinatorics. {V}olume 2}.
\newblock Cambridge University Press, 1999.

\bibitem{Tu1}
W.~T. Tutte.
\newblock A census of planar maps.
\newblock {\em Canad. J. Math.}, 15:249--271, 1963.

\bibitem{Tu3}
W.~T. Tutte.
\newblock On the enumeration of planar maps.
\newblock {\em Bull. Amer. Math. Soc.}, 74:64--74, 1968.

\bibitem{Walsh71}
T.~R.~S. Walsh.
\newblock {\em Combinatorial enumeration of non-planar maps}.
\newblock PhD thesis, University of Toronto, 1971.

\bibitem{Walsh83b}
T.~R.~S. Walsh.
\newblock Generating nonisomorphic maps without storing them.
\newblock {\em SIAM J. Alg. Disc. Meth.}, 4:161--178, 1983.

\bibitem{W13}
T.~R.~S. Walsh.
\newblock Counting maps on doughnuts.
\newblock {\em Theor. Comput. Sci.}, 502:4--15, 2013.

\bibitem{Walsh14}
T.~R.~S. Walsh.
\newblock A program that counts rooted and unrooted maps by genus, number of
  edges and number of vertices, 2014.
\newblock
  \url{http://www.info2.uqam.ca/%7Ewalsh_t/programs/map_counting_program.txt}.

\bibitem{WG14}
T.~R.~S. Walsh and A.~Giorgetti.
\newblock Efficient enumeration of rooted maps of a given orientable genus by
  number of faces and vertices.
\newblock {\em Ars Mathematica Contemporanea}, 7:263--280, 2014.

\bibitem{WGM12}
T.~R.~S. Walsh, A.~Giorgetti, and A.~Mednykh.
\newblock Enumeration of unrooted orientable maps of arbitrary genus by number
  of edges and vertices.
\newblock {\em Discrete Mathematics}, 312(17):2660--2671, 2012.

\bibitem{W1}
T.~R.~S. Walsh and A.~B. Lehman.
\newblock Counting rooted maps by genus {I}.
\newblock {\em J. Combin. Theory Ser. B}, 13:192--218, 1972.

\end{thebibliography}
\end{document}